\documentclass[12pt]{article}

\usepackage{amssymb,amsmath,amsthm,sectsty,url,mathrsfs,graphicx}
\usepackage[letterpaper,hmargin=1.0in,vmargin=1.0in]{geometry}
\usepackage[dvips,colorlinks,linkcolor=black,citecolor=black,filecolor=black,urlcolor=black]{hyperref}
\usepackage{color}
\sectionfont{\large}
\subsectionfont{\normalsize}
\numberwithin{equation}{section}
\newtheorem{maintheorem}{Theorem}[section]
\newtheorem{theorem}{Theorem}[section]
\newtheorem{lemma}[theorem]{Lemma}
\newtheorem{proposition}[theorem]{Proposition}
\newtheorem{corollary}[theorem]{Corollary}
\newtheorem{fact}[theorem]{Fact}
\newtheorem{definition}[theorem]{Definition}
\newtheorem{remark}[theorem]{Remark}

\newtheorem{claim}[theorem]{Claim}

\newtheorem{example}[theorem]{Example}
\newtheorem*{open}{Question}

\newcommand{\Var}{{\rm Var}}

\newcommand{\PP}{\mathscr{P}}
\newcommand{\laz}{\mathrm{L}}
\newcommand{ \tl}{t_{\mathrm{L}} }
\newcommand{\dl}{d_{\mathrm{L}}}
\newcommand{\ct}{\mathrm{c}}
\newcommand{\dc}{d_{\mathrm{c}}}
\renewcommand{\Pr}{ \mathrm P}

\newcommand{ \av}{ \mathrm{ave} }

\newcommand{ \mix}{ t_{\mathrm{c}} }
\newcommand{ \ave}{ t_{\mathrm{ave}} }

\newcommand{ \h}{ \mathrm{H} }
\newcommand{ \TV}{ \mathrm{TV} }

\renewcommand{\epsilon}{\varepsilon}

\DeclareMathSymbol{\leqslant}{\mathalpha}{AMSa}{"36} % nicer `smaller or equal'
\DeclareMathSymbol{\geqslant}{\mathalpha}{AMSa}{"3E} % nicer `larger or equal'
\DeclareMathSymbol{\eset}{\mathalpha}{AMSb}{"3F}     % nicer `emptyset'
\renewcommand{\leq}{\;\leqslant\;}                   % redef. of < or =
\newcommand{\N}{\mathbb N}
\newcommand{\R}{\mathbb R}
\newcommand{\Z}{\mathbb Z}

\usepackage[normalem]{ulem}

\begin{document}

% Title Page
%\begin{titlepage}
\title{The power of averaging at two consecutive time steps: Proof of a mixing conjecture by Aldous and Fill}
\author{Jonathan Hermon
\thanks{
Department of Statistics, UC Berkeley, USA. E-mail: {\tt jonathan.hermon@stat.berkeley.edu}.}
\and
Yuval Peres
\thanks{
Microsoft Research, Redmond, Washington, USA. E-mail: {\tt peres@microsoft.com}.}
}
\date{}
%\date{\today}
\maketitle

\begin{abstract}
Let $(X_t)_{t = 0 }^{\infty}$ be an irreducible reversible discrete-time Markov chain on a finite state space $\Omega $. Denote its transition matrix by $P$. To avoid periodicity issues (and thus ensuring convergence to equilibrium) one often considers the continuous-time version of the chain $(X_t^{\mathrm{c}})_{t \ge 0} $ whose kernel is given by $H_t:=e^{-t}\sum_k (tP)^k/k! $. Another possibility is to consider the associated averaged chain  $(X_t^{\mathrm{ave}})_{t
= 0}^{\infty}$, whose distribution at time $t$ is obtained by replacing $P^{t}$ by $A_t:=(P^t+P^{t+1})/2$. 

 A sequence of Markov chains is said to exhibit (total-variation) cutoff if
the convergence to stationarity in total-variation distance is abrupt. Let $(X_t^{(n)})_{t = 0 }^{\infty}$ be a sequence of irreducible reversible discrete-time Markov chains. In this work we prove that the sequence of associated continuous-time chains  exhibits total-variation cutoff around time $t_n$ iff the sequence of the associated averaged chains exhibits total-variation cutoff around time $t_n$. Moreover, we show that the width of the cutoff window for the sequence of associated averaged chains is at most that of the sequence of associated continuous-time  chains. In fact, we establish more precise quantitative relations between the mixing-times of the continuous-time and  the averaged versions of a reversible Markov chain, which provide an affirmative answer to a problem raised by Aldous and Fill (\cite[Open Problem 4.17]{aldous2000reversible}).    
\end{abstract}

%Keywords
\newpage

% Body

%%%%%%%%%%%%%%%%%%%%%%%%%%%
% Section 1: Introduction %
%%%%%%%%%%%%%%%%%%%%%%%%%%%
\section{Introduction}

Generically, we shall denote the state space of a Markov chain by $\Omega
$ and its stationary distribution by $\pi$. We say that the chain is finite, whenever $\Omega$ is finite.
Let $(X_t)_{t=0}^{\infty}$ be an irreducible Markov chain on a finite state
space $\Omega$ with transition matrix $P$ and stationary distribution $\pi$. We denote such a chain by $(\Omega,P,\pi)$.
A chain $(\Omega,P,\pi) $ is called \emph{\textbf{reversible}}
if $\pi(x)P(x,y)=\pi(y)P(y,x)$, for all $x,y \in \Omega$.

\medskip

We call a chain \textbf{\emph{lazy}},
if $P(x,x) \ge 1/2$ for all $x \in \Omega$. To avoid periodicity and near-periodicity issues, one often considers
the lazy version of a discrete time Markov chain, $(X_t^{\mathrm{L}})_{t
= 0}^{\infty}$, obtained by replacing $P$ with $P_{\laz}:=\frac{1}{2}(I+P)$. Periodicity issues can be avoided also by considering the continuous-time version of the chain, $(X_t^{\mathrm{c}})_{t
\ge 0}$.
This is a continuous-time Markov chain whose heat kernel is defined by $H_t(x,y):=\sum_{k=0}^{\infty}\frac{e^{-t}t^k}{k!}P^k(x,y)$.
It is a classic result of probability theory that for any initial condition
the distribution 
of both $X_t^{\mathrm{L}}$ and $X_t^{\mathrm{c}}$ converge to $\pi$ when $t$ tends to infinity. The object of the
theory of Mixing times of Markov chains is to
study the characteristic of this convergence (see \cite{levin2009markov}
for a self-contained
introduction to the subject).

\medskip

Since reversible Markov chains can only have period 2, one may wonder whether it suffices
to average over two consecutive times in order to avoid near-periodicity issues.  
This motivates considering the following Markov chain. For any $t \ge 0$, denote $A_t:=(P^t+P^{t+1})/2$. The \emph{\textbf{averaged chain}},
 $(X_t^{\mathrm{ave}})_{t
= 0}^{\infty}$,  with ``initial
state" $x$, is a  Markov chain, whose distribution at time $t
\ge 0$ is $A_t(x,\cdot)$, where $A_{t}(x,y):=
(P^{t}(x,y)+P^{t+1}(x,y))/2$. Equivalently, $(X_t^{\mathrm{ave}})_{t
= 0}^{\infty}:=(X_{t+\xi})_{t
= 0}^{\infty}$, where $\xi$ is a  $\mathrm{Bernoulli}(1/2)$ random variable,
 independent of $(X_t)_{t
= 0}^{\infty}$. In other words, if $X_0 \sim \mu $, the averaged chain either starts at a random position distributed according to $\mu $ (i.e.~it starts ``at time 0") or at a random position distributed as $\sum \mu(x) P(x,\cdot) $ (i.e.~it starts ``at time 1") with equal probability. After this, the averaged chain evolves according to the transition matrix $P$. The first to investigate the averaged chain were Peres and Sousi \cite{peres2013mixing}.
We review their results in the related work section.

\medskip

A sequence of Markov chains is said to exhibit (total-variation) cutoff
if
the convergence to stationarity in total-variation distance is abrupt (throughout we consider cutoff only in total-variation). In
this work we prove that given a sequence of irreducible  reversible finite discrete-time Markov chains, the sequence of associated continuous-time chains exhibits total-variation cutoff around time $t_n$ iff the sequence
of the associated averaged chains exhibits total-variation cutoff around
time $t_n$. See Corollary \ref{cor: cutoff} for a precise statement (we defer the formal definition of cutoff to the paragraph preceding Corollary \ref{cor: cutoff}).  In fact, we establish more precise quantitative relations between
the mixing times of the continuous-time and of the averaged versions of a
reversible discrete-time Markov chain (namely, Theorem \ref{thm: AF} and Proposition \ref{prop: easydirectionintro}), which provide an affirmative answer to a problem
raised by Aldous and Fill (\cite[Open Problem 4.17]{aldous2000reversible}, stated below). Moreover, we use them to deduce that when cutoff occurs, the width of the cutoff window for the sequence of associated averaged chains is at most that of the sequence of associated continuous-time chains (see Theorem \ref{thm: window} for a precise statement). 

\medskip

We denote by $\Pr_{\mu}^t$ (resp.~$\Pr_{\mu}$) the distribution of $X_t$ (resp.~$(X_t)_{t
= 0}^{\infty}$), given that the initial distribution is $\mu$. Similarly, we denote by $\h_{\mu}^t$ (resp.~$\h_{\mu}$) the distribution of $X_t^{\mathrm{c}}$ (resp.~$(X_t^{\mathrm{c}})_{t
\ge 0}$) given that  $X_0^{\ct} \sim \mu$. Finally, we denote
by $\mathrm{P}_{\mathrm{L},\mu}^{t}$ (resp.~$\mathrm{P}_{\mathrm{L},\mu}$)
the distribution of $X_t^{\mathrm{L}}$ (resp.~$(X_t^{\mathrm{L}})_{t
= 0}^{\infty}$), given that  $X_0^{\mathrm{L}}\sim \mu$.  When $\mu(\cdot)=1_{\cdot = x}$, for some $x \in \Omega$, we simply write $\Pr_x^t$ (similarly, $\h_x^t$, and $\Pr_{\laz,x}^t$) and $\Pr_x$ (similarly, $\h_x$ and $\Pr_{\laz,x}$).

 We denote the set of  distributions on a (finite) set $\Omega$ by
$\mathscr{P}(\Omega) $. For any $\mu,\nu \in \mathscr{P}(\Omega) $,  their \emph{\textbf{total-variation
distance}} is defined as
$$\|\mu-\nu\|_\mathrm{TV} := \frac{1}{2}\sum_{x \in \Omega} |\mu(x)-\nu(x)|=\max_{B \subset \Omega }\mu(B)-\nu(B) .$$
The worst-case total-variation distance at time $t$ of the continuous-time (resp.~lazy) chain is defined as $$\dc(t)
:= \max_{x \in \Omega} \dc(t,x)\quad \text{ (respectively, } \dl(t)
:= \max_{x \in \Omega} \dl(t,x) ),$$ 
where for every $\mu \in \PP(\Omega)$, $$\dc(t,\mu) :=  \| \Pr_{\mu}(X_t^{\mathrm{c}}
\in \cdot)- \pi\|_\mathrm{TV}=\|\h_{\mu}^{t}- \pi\|_\mathrm{TV} \text{ and }$$ $$\dl(t,\mu) :=  \| \Pr_{\mu}(X_t^{\mathrm{L}}
\in \cdot)- \pi\|_\mathrm{TV}=\| \Pr_{\laz,\mu}^{t}- \pi\|_\mathrm{TV} .$$
The $\epsilon$\textbf{-mixing-time} of the continuous-time
(resp.~lazy) chain is defined as  $$t_{\mathrm{c}}(\epsilon) := \inf \left\{t : \dc(t) \leq
\epsilon \right\}, \quad t_{\mathrm{L}}(\epsilon) := \inf \left\{t
: d_{\mathrm{L}}(t) \leq
\epsilon \right\}. $$ 
We also define the corresponding $\epsilon$-mixing-times w.r.t.~initial distribution $\mu$ to be $$t_{\mathrm{c}}(\epsilon,\mu)
:= \inf \left\{t : d_{\ct}(t,\mu) \leq
\epsilon \right\}  \text{ and } \tl(\epsilon,\mu)
:= \inf \left\{t : \dl(t,\mu) \leq
\epsilon \right\} .$$
Similarly, for the averaged chain we define $d_\mathrm{ave}(t):=\max_{x\in \Omega }d_\mathrm{ave}(t,x)$, where $$d_\mathrm{ave}(t,\mu):=\left\| (\Pr_\mu^{t}+\Pr_\mu^{t+1})/2-
\pi \right\|_\mathrm{TV}=\|\mu(P^{t+1}+P^t)/2 -\pi \|_{\mathrm{TV}}.$$ The  $\epsilon$\textbf{-mixing-time} of the averaged chain (respectively, w.r.t.~$X_0 \sim \mu $, i.e.~w.r.t.~$X_0^{\av} \sim \frac{\mu(I+P)}{2} $)) is denoted by 
$$t_{\mathrm{ave}}(\epsilon):=\inf \left\{t : d_{\mathrm{ave}}(t) \leq
\epsilon \right\} \text{ (respectively, }t_{\mathrm{ave}}(\epsilon,\mu)
:= \inf \left\{t : d_{\av}(t,\mu) \leq
\epsilon \right\} ).$$
When $\epsilon=1/4$ we omit it from the above
notation.

\medskip

We
denote $\Z_{+}:=\{n \in \Z: n \ge 0 \}$ and $\R_+:=\{t \in \R:t \ge 0 \}$.
Let $\phi:\R_{+} \to \R_{+}$ and $\psi:(0,1]
\to (0,1] $. We write  $\phi (t)
\sim t$ if $\lim_{t \to \infty}\phi(t)/t=1$. We write $\psi=o(1)
$ if $\lim_{\epsilon \to 0}\psi(\epsilon)=0 $.
In \cite{aldous2000reversible} Aldous and Fill raised the following question:
\begin{open}[Open Problem 4.17 \cite{aldous2000reversible}]
Show that there exist $\psi:(0,1]
\to (0,1] $ and $\phi  :\R_{+} \to \Z_{+}$ satisfying $\psi=o(1)$ and $\phi(t) \sim t$ such that for every finite irreducible reversible Markov chain, $$ \forall t \ge 0, \quad d_\mathrm{ave}(\phi(t)) \le \psi (\dc(t)).$$
\end{open}
Our Theorem \ref{thm: AF2}, which is in fact a weaker version of our main result, Theorem \ref{thm: AF}, solves Aldous and Fill's Problem. Denote $a \vee b:=\max \{a,b\}$,  $a \wedge b:=\min
\{a,b\}$. For every $t \in \R$ we denote the ceiling of $t$ by $\lceil t \rceil:=\min \{z \in \Z : z \ge t \} $.
\begin{definition}
Let  $0 < \alpha < 1/2$, $C>0$,  $t \ge 1 $ and $x \in (0,1)$. We define
\[\psi_{\alpha,C}(x):= 1 \wedge ( x+ C 
|\log (2x)|^{-\alpha }) \quad \text{and} \quad \phi_{\alpha,C}(t):=t+\lceil C  t^{\frac{1+2\alpha}{2}}
\sqrt{\alpha \log t} \rceil.   \]

\end{definition}
 \begin{remark}
\label{rem: simplifying}
Note that $\phi_{\alpha,C}(t) \sim t$ and $\psi_{\alpha,C}=o(1)$, for all $C>0$ and $0<\alpha < 1/2$. 
%$$0 \le \phi_{\alpha,C}(t)/t-1 \le C \sqrt{\alpha}t^{-\frac{1-2\alpha}{2}}+2/t %\to 0, \quad \text{as }t \to \infty. $$
%$$\forall \ell \ge 1, \quad \psi_{\alpha}(e^{-\ell} /2) < e^{-\ell}/2 +C %\ell^{-\alpha}\to 0, \quad \text{as }\ell \to \infty.  $$
\end{remark}
\begin{maintheorem}
\label{thm: AF}
There exist absolute constants $C_{1},C_{2},C_3>0$ such that
for every finite irreducible reversible Markov chain,  $(\Omega,P,\pi)$,   $\mu \in \PP(\Omega)$, 
 $0 < \alpha \le 1/2$ and $t \ge 1 $,   \begin{equation}
\label{eq: main1}
\dl(\phi_{\alpha,C_1}(t),\mu) \le \dc(t/2,\mu)+  C_{2} t^{-\alpha} .
\end{equation}
\begin{equation}
\label{eq: main2}
d_{\av}(\phi_{\alpha,C_{1}}(t),\mu) \le \dl(2t,\mu)+C_{2} t^{-\alpha}.
\end{equation}
%Consequently, there exists an absolute constant $C_{3}>0$ such that
\begin{equation}
\label{eq: main3}
d_{\av}(\phi_{\alpha,C_{3}}(t),\mu)\le \dc(t,\mu)+  2C_{2} t^{-\alpha} .
\end{equation}
Moreover, (\ref{eq: main1})-(\ref{eq: main3}) remain valid when $\mu$ is omitted from both sides.
\end{maintheorem}
Note that \eqref{eq: main3} follows from (\ref{eq: main1})-(\ref{eq: main2}) by picking $C_3$ so that $\phi_{\alpha,C_{3}}(t) \ge \phi_{\alpha,C_{1}} (\lceil \frac{1}{2} \phi_{\alpha,C_{1}}(2t) \rceil )$. 
\begin{remark}
The converse inequality $\dc(t+2t^{3/4}) \le d_{\av}(t)+ e^{-\sqrt{t}}$ is easy ((\ref{eq: ctandlazyarefaster1intro})). Combined with \eqref{eq: main3} one can readily see that $\dc(\cdot)$  exhibits an abrupt transition iff $d_{\av}(\cdot)$ exhibits an abrupt transition (in which case, both occur around the same time).
\end{remark}
\begin{maintheorem}
\label{thm: AF2}
There exist absolute constants $C_{1},C_{2}>0$ such that for every finite irreducible reversible Markov chain
\begin{equation}
\label{eq: AFthm1}
d_{\av}(\phi_{\alpha,C_{1}}(t)) \le  \psi_{\alpha,C_{2}} (\dc(t)), \text{ for every }0 < \alpha < 1/2 \text{ and } t \ge 2.
\end{equation}
\end{maintheorem}
\begin{remark}
\label{rem: alternativephrasingofAF}
Theorem \ref{thm: AF2} can be rephrased as follows. There exist absolute constants $C_{1},C_{2}>0$ such that for every finite irreducible reversible Markov chain,
\begin{equation}
\label{eq: AFalternative}
\ave(\psi_{\alpha,C_{2}}(\epsilon)) \le  \phi_{\alpha,C_{1}} (\mix(\epsilon)), \text{ for all }0<\alpha<1/2 \text{ and } 0<\epsilon < 1.
\end{equation}
\end{remark}
Theorem \ref{thm: AF2} is an immediate consequence of (\ref{eq: main3}) together with   the ``worst-case" estimate $\dc(t) \ge (e^{-2t}/2)1_{|\Omega|>1} $ (e.g.~\cite[Lemma 20.11]{levin2009markov}). We omit the details.
Theorem \ref{thm: AF} follows in turn as the particular case  $s:=2 \vee t^{\alpha}\sqrt{ \alpha \log t}  $  of the following proposition.
\begin{proposition}
\label{prop: AFintro}
There exists an absolute constant $C$ such that for every finite irreducible reversible chain,  $(\Omega,P,\pi)$,
every $\mu \in \PP(\Omega)$, 
$t \ge 2$ and $ s \in [2, e^{t}]$ we have that
\begin{equation}
\label{eq: cttolazyintro}
\dl(t+ \lceil s \sqrt{t}\rceil , \mu )\le\dc(t/2, \mu )+
Cs^{-1} \sqrt{\log s  } .
\end{equation}
\begin{equation}
\label{eq: lazytoaveintro}
d_{\av}(t+\lceil s \sqrt{t}\rceil, \mu )\le \dl(2t, \mu )+ 
Cs^{-1} \sqrt{\log s} .
\end{equation}
\end{proposition}
We now make two remarks regarding the sharpness of \eqref{eq: lazytoaveintro}. The first concerns the error term $Cs^{-1} \sqrt{\log s}$ (and also the ``error term", $\psi_{\alpha,C_2}(\dc(t))-\dc(t)$, in \eqref{eq: AFthm1}). The second concerns the ``time-shift" term $\lceil s\sqrt{t} \rceil $. 
\begin{remark}
Denote $s=s_{n,\alpha}:=\lceil n^{0.5+\alpha}\rceil  \text{ and }t=t_{n,\alpha}:=4n+ s$. In \S~\ref{sec: examples} we construct for every $0<\alpha \le 1/2$  a sequence of chains with $\mix^{(n)}=(4 \pm o(1))n$ such that for some absolute constants $c_1,c_{2}>0$ the $n$-th chain in the sequence satisfies that \begin{equation}
\label{eq: AFremarkintro2}
d_{\av}(t+s )-\dc(t) \ge \frac{c_1}{s }   \ge \frac{c_2}{\left[ \log (1/\dc(t))\right]^{\frac{1+2\alpha}{4\alpha}}}.  
\end{equation} 
 Thus the inverse polynomial decay (w.r.t.~$s$) in (\ref{eq: lazytoaveintro}) is the correct order of decay, up to the value of the exponent.
\end{remark}

\begin{remark}
When $s$ is fixed, the ``time-shift" term $s \sqrt{t} $ in \eqref{eq: lazytoaveintro} is of order $\sqrt{t}$. This cannot be improved. To see this, consider a birth and death chain on $[n]:=\{1,2,\ldots,n\}$ with $P(i+1,i)=e^{-n}=1-P(i+1,i+2)$ for $i \in [n-2]$ and $P(1,2)=1=P(n,n-1)$. Then if $r_n=o(\sqrt{n})$ we have that $\dl(2n-r_{n})=1/2 \pm o(1)$, while  $d_{\av}(n-3)=1-o(1)$.
\end{remark}

The following proposition offers a converse to Theorem \ref{thm: AF}. The
argument in the proof of (\ref{eq: ctandlazyarefaster1intro}) is  due to Peres and Sousi (\cite[Lemma 2.3]{peres2013mixing}).\begin{proposition}
\label{prop: easydirectionintro}
Let $(\Omega,P,\pi)$ be a finite irreducible  Markov chain. Then for every $t \in \N$,  $0<s \le \sqrt{t}$ and $\mu \in \PP(\Omega)$,
\begin{equation}
\label{eq: ctandlazyarefaster1intro}
\begin{split}
 \dc(t+s \sqrt{t},\mu)  
&  \le d_{\av}(t,\mu)+e^{-
s^{2}/4}. 
\\  \dl(2t+ \lceil 2s \sqrt{t}\rceil,\mu) & \le d_{\av}(t,\mu)+e^{-
s^{2}/4}.  
\end{split}
\end{equation}
\begin{equation}
\label{eq: ctfasterthanlazy1intro}
 \dc(t+s \sqrt{t},\mu) 
\le d_{\laz}(2t,\mu)+e^{- s^{2}/2}.  
\end{equation}
\end{proposition}

\begin{remark}
\label{rem: abelian}
In \cite{Widder} p.~195, it is written: "a theorem is Abelian if it says something about an average of a sequence from a hypothesis about its ordinary limit; it is Tauberian if conversely the implication goes from average to limit".

 Proposition \ref{prop: easydirectionintro}  is easier and more general than our Theorem \ref{thm: AF} (as it does not assume reversibility) because it is an Abelian theorem, while our Theorem \ref{thm: AF} is Tauberian, hence requires the reversibility assumption, as we now demonstrate.
One (non-reversible) instance in which (\ref{eq: lazytoaveintro}) fails is a biased random walk on the $n$-cycle with $P(i,i-1)=n^{-\ell}=1-P(i,i+1)$, where $i-1$ and $i+1$ are defined modulo $n$ and $\ell>0$ is arbitrary. In this example $\tl(\epsilon)/(n^2 | \log  \epsilon| ) =\Theta (1) $, however $\ave (\epsilon)/(n^{\ell+2} | \log  \epsilon| ) = \Theta (1)$ (uniformly in $\epsilon \in (0,1/2] $). 
\end{remark}
Next, consider a sequence of such chains, $((\Omega_n,P_n,\pi_n): n \in \N)$,
each with its corresponding
worst-distance from stationarity $d_n(t)$, its mixing-time $t_{\mathrm{c}}^{(n)}$,
etc..
Loosely speaking, the (total-variation) \emph{\textbf{cutoff phenomenon}}
 occurs when over a negligible period of time, known as the \emph{\textbf{cutoff
window}}, the worst-case total-variation distance  drops abruptly from a value
close to 1 to near $0$. In other words, one should run the $n$-th chain until
the cutoff time for it to even slightly mix in total-variation, whereas
running it any further is essentially redundant.
Formally, we say that a sequence of chains exhibits a \emph{\textbf{continuous-time
cutoff}}
if the
following
sharp transition in its convergence to stationarity occurs:
$$\lim_{n \to \infty}t_{\mathrm{c}}^{(n)}(\epsilon)/t_{\mathrm{c}}^{(n)}(1-\epsilon)=1, \quad
\text{for every }0<\epsilon <1.$$
We say that a sequence of chains exhibits an \emph{\textbf{averaged
cutoff}} (resp.~\emph{\textbf{lazy
cutoff}})
if
$$\lim_{n \to \infty}t_{\mathrm{ave}}^{(n)}(\epsilon)/t_{\mathrm{ave}}^{(n)}(1-\epsilon)=1\, \text{ (resp., }\lim_{n \to \infty}t_{\mathrm{L}}^{(n)}(\epsilon)/t_\mathrm{L}^{(n)}(1-\epsilon)=1),
\text{ for every }0<\epsilon <1.$$
The following corollary follows at once from Theorem \ref{thm: AF} together with Proposition \ref{prop: easydirectionintro}.
\begin{corollary}
\label{cor: cutoff}
Let $(\Omega_{n},P_{n},\pi_{n})$ be a sequence of finite irreducible reversible Markov chains. Then the following are equivalent
\begin{itemize}
\item[(i)] The sequence exhibits a continuous-time cutoff.

\item[(ii)] The sequence exhibits a lazy cutoff.

\item[(iii)] The sequence exhibits an averaged cutoff.
\end{itemize}
Moreover, if (i) holds, then $\lim_{n \to \infty}t_{\mathrm{ave}}^{(n)}/\mix^{(n)}=\lim_{n \to \infty} t_{\mathrm{L}}^{(n)}/(2\mix^{(n)})=1$.
\end{corollary}
\begin{remark}
The equivalence between (i) and (iii) was previously unknown. In \cite{chen2013comparison} it was shown that (i) and (ii) are equivalent even without the assumption of reversibility.
\end{remark}
Our last point of comparison  is related to the width of the
\emph{\textbf{cutoff window}}. We say that a sequence of chains exhibits a continuous-time (resp.~averaged) cutoff with a cutoff window 
$w_n$ if $w_n=o(\mix^{(n)})$ (resp.~$w_n=o(\ave^{(n)})$) and for every $0<\epsilon \le 1/4$ there exists some constant $C_{\epsilon}>0$ (depending only on $\epsilon$) such that 
$$\forall n, \quad \mix^{(n)}(\epsilon)-\mix^{(n)}(1-\epsilon) \le C_{\epsilon}w_n \quad (\text{resp. }\ave^{(n)}(\epsilon)-\ave^{(n)}(1-\epsilon) \le C_{\epsilon}w_n).$$
One can define the notion of a cutoff window for a sequence of associated lazy chains in an analogous manner. Note that the window defined in this manner is not unique.
\begin{maintheorem}
\label{thm: window}
Let $(\Omega_{n},P_{n},\pi_{n})$ be a sequence of finite irreducible reversible Markov chains.
\begin{itemize}
\item[(i)]
 Assume that the sequence exhibits a continuous-time cutoff with a window $w_n$. Then it  exhibits also an averaged cutoff with a window $w_n$.
\item[(ii)] Assume that the sequence exhibits an averaged cutoff with a window $w_n$. Then it exhibits also a continuous-time cutoff with a window $w'_n:=w_n \vee \sqrt{\mix^{(n)}} $. 
\end{itemize}
\end{maintheorem}
Theorem \ref{thm: window} follows easily from Propositions \ref{prop: AFintro} and \ref{prop: easydirectionintro} in conjunction with the following result. We prove Theorem \ref{thm: window} in \S~\ref{s: 5} for the sake of completeness.
\begin{proposition}[\cite{chen2013comparison} Chen and Saloff-Coste]
\label{prop: sqrt}
Let $(\Omega_{n},P_{n},\pi_{n})$ be a sequence of finite irreducible reversible Markov chains. The sequence exhibits a continuous-time cutoff with a window $w_n$ iff it exhibits a lazy  cutoff with a window $w_n$, in which case $w_n = \Omega \left( \sqrt{\mix^{(n)}} \right) $.
\end{proposition} 
\begin{remark}
There are cases in which the cutoff window for the sequence of the associated averaged chains can be much smaller than that of the associated continuous-time chains. For instance, let $G_n$ be a sequence of random $n$-vertex $d_n$-regular graphs, for some $d_n$ such that $\log n \ll d_n \log d_n =n^{o(1)} $. Let $(X_t^{(n)})_{t \in \Z_+ }$ be the sequence of discrete-time simple random walks on $G_n$. Then \cite{LS} w.h.p.~(i.e.~with probability $1-o(1)$, over the choice of the graphs) \[|t_{\mathrm{ave}}^{(n)}(\epsilon)-\lceil \log_{d_{n}-1}(d_{n}n) \rceil | \le 1 , \quad \text{ for every }\epsilon \in (0,1),\] while the cutoff window for the sequence of associated  continuous-time chains is $\sqrt{ \log_{d_{n}-1}n} $.  
\end{remark}
 \subsection{Related work}
\label{s:related}

This work was greatly motivated by the results of Peres and Sousi in \cite{peres2013mixing} about the averaged chain.  Their approach relied on the theory of random times to stationarity combined with a certain ``de-randomization" argument which shows that for every finite irreducible reversible Markov chain and every stopping time $T$ such that $X_{T} \sim \pi$, $t_{\mathrm{ave}} \le 220\max_{x \in \Omega} \mathbb{E}_x[T]$. As a  consequence, they showed that for all $\alpha \in (0, 1/2)$ (this was extended to $\alpha=1/2$ in \cite{griffiths2014tight}), there exist  constants $c_{\alpha},c'_{\alpha}>0$ such that for every lazy finite irreducible reversible chain \[c'_\alpha t_{\mathrm{H}}(\alpha) \leq t_{\mathrm{ave}} \leq c_\alpha t_{\mathrm{H}}(\alpha), \quad \text{where} \]
\[ t_{\mathrm{H}}(\alpha):=\max_{x
\in \Omega,A \subset \Omega :\,\pi(A)
\ge \alpha}\mathbb{E}_{x}[T_{A}] \quad \text{and} \quad T_{A}:=\inf\{t:X_t \in A \}.\] 
Using this, they showed that there exist some absolute constants $c_1,c_2>0$ such that $$c_1\tl \leq t_{\mathrm{ave}}\leq c_{2}\tl .$$  Implicitly, they showed that $\text{for every }0<\epsilon \le 1/4 \text{ and }0 < \alpha \le 1/2$, $$t_{\mathrm{ave}}(\epsilon) \leq c_{\alpha}\epsilon^{-4}t_{\mathrm{H}}(\alpha) .$$ This was the first progress towards resolving Aldous and Fill's Open Problem. Alas, this is too coarse for the purpose of resolving it.

\medskip

Our approach, which is somewhat
similar to that taken in \cite{cutoff}, is more direct than that taken in \cite{peres2013mixing}. As in \cite{cutoff}, where  Starr's maximal inequality was used to obtain a characterization of the cutoff phenomenon for reversible Markov chains, the key ingredient in the proof of Proposition  \ref{prop: AFintro} is a maximal inequality, due to Stein \cite{stein1961maximal} \eqref{eq: Steinmi}.

\section{A maximal inequality}
\label{s:trelimplications}
In this section we state maximal inequalities which shall be utilized in the proof of the main results. We start with a few basic definitions.\begin{definition}
\label{def: L_p distance of measures}
Let $(\Omega,P,\pi)$ be a finite reversible chain. For $f \in \R^{\Omega}$, let $$\mathbb{E}_{\pi}[f]:=\sum_{x \in \Omega}\pi(x)f(x) \quad \text{ and } \quad \Var_{\pi}f:=\mathbb{E}_{\pi}[(f-\mathbb{E}_{\pi}f)^{2}].$$ The inner-product $\langle \cdot,\cdot \rangle_{\pi}$ and $L^{p} $ norm are
$$\langle f,g\rangle_{\pi}:=\mathbb{E}_{\pi}[fg] \text{ and } \|f \|_p:=\left( \mathbb{E}_{\pi}[|f|^{p}]\right)^{1/p},\, 1 \le p < \infty.$$
We identify $P^t$, $P_{\mathrm{L}}^t$, $A_t$,  $H_t$ with the linear operators on $L^p(\R^{\Omega},\pi)
$ given by  
\begin{equation*}
\begin{split}
 A_t f(x):=\sum_{y \in \Omega}A_{t}(x,y)f(y)=& \mathbb{E}_{x}[f(X_t^{\mathrm{ave}})] \text{,  }  H_t f(x):=\sum_{y \in \Omega}H_{t}(x,y)f(y) =\mathbb{E}_{x}[f(X_t^{\mathrm{c}})],
\\ P^tf(x):=\mathbb{E}_{x}[f(X_t)]&\text{ and } P^{t}_{\mathrm{L}}f(x):=\sum_{y \in \Omega
}P_{\mathrm{L}}^{t}(x,y)f(y) =\mathbb{E}_{x}[f(X_t^{\mathrm{L}})]. 
\end{split}
\end{equation*}
By reversibility $P^t$, $P_{\mathrm{L}}^t$, $A_t,H_t:L^2 \to L^2$ are all self-adjoint (w.r.t.~$\langle \cdot,\cdot \rangle_{\pi}$).
\end{definition}

\begin{definition}
Let $P$ be a linear operator and $k \in \Z_+$.
We define $\bigtriangleup P^k:=P^{k+1}-P^k=P^k(P-I)$. For $r > 1$, we define inductively 
 $ \bigtriangleup ^r P^k:= \bigtriangleup ( \bigtriangleup ^{r-1}
P^k)= \bigtriangleup ^{r-1}
P^{k+1}- \bigtriangleup ^{r-1}
P^k=P^k(P-I)^{r}$. Similarly, we define $\bigtriangleup A_k:=
A_{k+1}-A_k=\frac{1}{2}P^k(P^2-I) $. 
\end{definition}
Let $(\Omega,\mu)$ be a probability space. Let $P: L^2(\Omega,\mu) \to L^2(\Omega,\mu)  $ be a positive (i.e.~$f \ge 0 \Longrightarrow Pf \ge 0$) self-adjoint linear operator whose spectrum is contained in the interval $[0,1]$. It is  noted in \cite{stein1961maximal}
that for all $r \ge 1$, there exists
a constant $C_r$ (independent of $(\Omega,\mu)$ and $P$), such that for every $f \in L^2(\Omega,\mu)  $
\begin{equation}
\label{eq: Steinmi}
\|\sup_{t \ge 0} (t+1)^r  \bigtriangleup^r P^t f \|_2
\le C_r\|f\|_2.
\end{equation}
 In \cite{maximal} Stein's argument is extended to the setup where $P$ is a positive contraction with $M(P):=\sup_t t \|P^{t+1}-P^t \|_2 < \infty $ without  the assumptions that $P$ is self-adjoint and that its spectrum is contained in $[0,1]$. In this more general setup $C_{r}$ depends also on $M(P)$. 
\begin{corollary}
\label{cor: maxtriangle}
 There exists an absolute constant $C$ such that for every finite irreducible reversible
Markov chain, $(\Omega,P,\pi)$ and every   
 $f \in \R^{\Omega}$  
\begin{equation}
\label{eq: maxtraingle1}
\left\|\sup_{t \ge 0 }(t+1)\bigtriangleup P_{\mathrm{L}}^t f  \right\|_2^{2} \le C \Var_{\pi}  f \quad \text{and} \quad \left\|\sup_{t \ge 0 }(t+1)\bigtriangleup A_t  f  \right\|_2^{2} \le
C \Var_{\pi}  f.
\end{equation}
\end{corollary}
\emph{Proof:}
Note that $\bigtriangleup A_{2t}f=\frac{P^{2t+2}-P^{2t}}{2}f=\frac{1}{2}\bigtriangleup
(P^2)^tf$ and $\bigtriangleup
A_{2t+1}f=\frac{1}{2}\bigtriangleup (P^2)^t(Pf)$. Hence \eqref{eq: maxtraingle1} follows from \eqref{eq: Steinmi} applied to $P_{\mathrm{L}}$ and $P^2$ by noting that $\bigtriangleup P_{\mathrm{L}}^t f=\bigtriangleup P_{\mathrm{L}}^t (f-\mathrm{E}_{\pi}[f]) $,   $\bigtriangleup A_t f=\bigtriangleup A_t (f-\mathrm{E}_{\pi}[f]) $ and $\Var_{\pi}  (Pf) \le \Var_{\pi}  f $. \qed   
\section{Proof of Proposition \ref{prop: AFintro}.}
\label{sec: applications}
In this section we prove Proposition \ref{prop: AFintro}. As noted in the introduction, Theorem \ref{thm: AF} follows as a particular case of Proposition \ref{prop: AFintro} and Theorem \ref{thm: AF2}, in turn, follows in a trivial manner from Theorem \ref{thm: AF}.
We now state large deviation estimates for the Poisson and Binomial distributions. For a proof see e.g.~\cite[Appendix A]{alon2004probabilistic}. \begin{fact}
\label{fact: LD}
Let $Y \sim \mathrm{Pois}(\mu)$ and let $Y' \sim \mathrm{Bin}(t,1/2) $.
Then for every $\epsilon > 0$ we have
that
\begin{equation}
\label{eq: LD1}
\begin{split}
&\mathbb{P}[Y \le \mu(1-\epsilon)] \le e^{- \epsilon^{2} \mu/2}, \quad \mathbb{P}[Y
\ge \mu(1+\epsilon)] \le \exp \left(-\frac{\epsilon^{2} \mu}{2(1+\epsilon/3)}\right), \\ & \mathbb{P}[Y'
\le t(1-\epsilon)/2] =\mathbb{P}[Y'
\ge t(1+\epsilon)/2]  \le e^{- \epsilon^{2} t/4}. 
 \end{split}
 \end{equation}
\end{fact}

Let  $(N(t))_{t\ge 0}$ and $(M(t))_{t\ge 0}$ be homogeneous Poisson processes with rate 1, such that  $(N(t))_{t\ge 0}$ , $(M(t))_{t\ge
0}$ and $(X_t)_{t
= 0}^{\infty}$ are mutually independent. We define
\[N_{\laz}(t):=N(t)+M(t) \text{ and }S(\ell):=\sum_{k=1}^{\ell}q_k \sim \mathrm{Bin}(\ell,1/2), \]
\[ \text{where} \quad  q_{k}:=1_{N(T_{k})>N(T_{k-1})} \quad \text{and} \quad T_k:=\inf \{t:N_{\laz}(t)=k \}.\] Let $(\Omega,P,\pi)$ be a Markov chain. 
The \textbf{\emph{natural coupling}} of $(X_{t}^{\ct})_{t \ge 0}$,  $(X_{t})_{t \in \Z_+}$ and~$(X_{t}^{\mathrm{L}})_{t \in \Z_+}$ is defined by setting $X_{t}^{\mathrm{L}}:=X_{S(t)} $ and  $X_{t}^{\ct}:=X_{N(t)}=X_{N_{\laz}(t)}^{\mathrm{L}}$.
 
As can be seen from the natural coupling, $H_t=\sum_{k \ge 0}\frac{e^{-2t}(2t)^k}{k!}P_{\mathrm{L}}^k$. This also follows from Poisson thinning. Also, in the natural coupling ~$(X_{t}^{\mathrm{L}})_{t \in \Z_+}$
and $(N_{\laz}(t))_{t \ge 0}$ are independent. The same holds for  $(X_{t})_{t
\in \Z_+}$ and
 $(S(t))_{t=0}^{\infty}$. The next lemma follows from the natural coupling by a standard construction (cf.~the proofs of Proposition 4.7
and Theorem 5.2 in \cite{levin2009markov}). 
\begin{lemma}
\label{lem: coupling}
Let $(\Omega,P,\pi)$ be a finite irreducible Markov chain. Let $\mu \in \PP(\Omega)$ and  $t \in \R_+$. \begin{itemize}
\item[(1)]
 There exists a coupling $((Y_i^{\mathrm{L}})_{i \in \Z_+},(Z_i^{\mathrm{L,\pi}})_{i \in \Z_+},\xi_{t} )$, such that $(Y_i^{\mathrm{L}})_{i \in \Z_+ } \sim \Pr_{\mathrm{L},\mu} $, $(Z_i^{\mathrm{L,\pi}})_{i \in \Z_+} \sim \Pr_{\mathrm{L},\pi}
$ (the law of the stationary lazy chain), $\xi_{t} \sim \mathrm{Pois}(t) $ in which $\xi_{t} $ and  $(Y_i^{\mathrm{L}})_{i
\in \Z_+ } $ are independent and $$\Pr[Y_{\xi_{t}}^{\mathrm{L}} = Z_{0}^{\mathrm{L,\pi}}  ]=\Pr[Y_{\xi_{t}+i}^{\mathrm{L}} = Z_{i}^{\mathrm{L,\pi}} \text{ for all }i \ge 0 ]=1-\dc(t/2,\mu).$$
\item[(2)] There exists a coupling $((Y_i)_{i \in \Z_+},(Z_i^{\pi})_{i
\in \Z_+},\xi'_{t} )$, such that $(Y_i)_{i
\in \Z_+ } \sim \Pr_{\mu} $, $(Z_i^{\pi})_{i
\in \Z_+}\sim \Pr_{\pi} $ (the law of the stationary chain), $\xi'_{t} \sim \mathrm{Bin}(2t,1/2)
$ in which $\xi'_{t} $ and  $(Y_i)_{i
\in \Z_+ } $ are independent and $$\Pr[Y_{\xi'_{t}} = Z^{\pi}_{0}  ]=\Pr[Y_{\xi'_{t}+i} = Z^{\pi}_{i} \text{ for all
}i \ge 0]=1-\dl(2t,\mu).$$
\end{itemize}
\end{lemma}
\begin{definition}
\label{def: Dc}
Let $t  \ge 1$ and $s \in [2, e^{t}]$. Denote 
\begin{equation}
\begin{split}
r=r_{s,t}&:=2\sqrt{2t\log
s},
\\  J=J_{s,t}&:=[(t-r  )\vee 0,t+r ],
\\  m=m_{s,t}&:=\lceil r( \sqrt{s}+1)\rceil.
\end{split}
\end{equation} 
In the notation of Lemma \ref{lem: coupling} (with both couplings taken w.r.t.~time $t$), let $G$ be the event
that $Y_{\xi_{t}+i}^{\mathrm{L}} = Z_{i}^{\mathrm{L,\pi}} \text{ for all
}i \ge 0 $ and that $\xi_t \in J $.
Similarly, let $G'$ be the event
that $Y_{\xi'_{t}+i} = Z^{\pi}_{i} \text{ for all
}i \ge 0 $ and that $\xi'_t \in J$. 
\end{definition}
In the following proposition, we only care about (\ref{eq: cttolazy4}) and (\ref{eq: lazytoave4}) (which imply (\ref{eq: cttolazyintro}) and (\ref{eq: lazytoaveintro}), respectively; i.e.~the below proposition implies Proposition \ref{prop: AFintro}). We present the rest of the equations in order to make it clear that (\ref{eq: lazytoave4}) is obtained in an analogous manner to (\ref{eq: cttolazy4}). Thus, we shall only prove part (i) of Proposition \ref{cor: coupling}.

In
the notation of Definition \ref{def: Dc}, the term $\dc(t/2, \mu ) +2/s^{2}$ appearing in \eqref{eq: cttolazy1} and (\ref{eq: cttolazy4}) (resp.~$\dl(2t,\mu)+2/s^{2}$ appearing in \eqref{eq: lazytoave1} and (\ref{eq: lazytoave4})) is an upper bound on the probability that $G$ (resp.~$G'$) fails (where the term $2/s^{2}$ is obtained via Fact \ref{fact: LD}).
\begin{proposition}
\label{cor: coupling}
Let $(\Omega,P,\pi)$ be a finite irreducible reversible chain.
Let $\mu \in \PP(\Omega)$. Let  $B
\subset \Omega$.  Let $t \ge 1$ and $2 \le s \le e^{t}$. In
the notation of Definition \ref{def: Dc},\begin{itemize}
\item[(i)]  Let $\eta^{\mathrm{L}}:=1_{Y_{t+m}^{\mathrm{L}} \in B}$ and $\eta^{\mathrm{L,\pi}}:=1_{Z_m \in B}^{\mathrm{L,\pi}}$ (where $m=\lceil r( \sqrt{s}+1)\rceil$,  $r=2\sqrt{2t\log
s} $). Then
\begin{equation}
\label{eq: cttolazy1}
\pi(B)-\Pr_{\mu}[X_{t+m}^{\mathrm{L}} \in B ] \le  \frac{2}{s^{2}}  +\dc(t/2, \mu ) + \mathbb{E}[(\eta^{\mathrm{L,\pi}}-\eta^{\mathrm{L}})1_{G}] .
\end{equation}
\begin{equation}
\label{eq: cttolazy2}
|\mathbb{E}[(\eta^{\mathrm{L}}-\eta^{\mathrm{L,\pi}})1_{G}] |^{2} \le   s^{-1}
   \mathbb{E}_{\pi} \left[
\sup_{i \ge r \sqrt{s}  }i^{2}|\bigtriangleup P_{\mathrm{L}}^{i}1_B |^2 \right] \le Cs^{-1} \Var_{\pi}1_B \le \frac{C}{s} .
\end{equation}

Consequently,
\begin{equation}
\label{eq: cttolazy4}
\dl(t+m, \mu )\le  \dc(t/2, \mu )+\frac{2}{s^{2}}+\sqrt{C/s}
   .
\end{equation}
\item[(ii)] Let $w \sim \mathrm{Bernoulli}(1/2) $ be independent of $((Y_i)_{i \in
\Z_+},(Z^{\pi}_i)_{i
\in \Z_+},\xi'_{t} ) $. Let $ \eta=1_{Y_{t+m+w} \in B}$ and $ \eta^{\pi}=1_{Z^{\pi}_{m+w} \in B}$. Then
\begin{equation}
\label{eq: lazytoave1}
\pi(B)-\Pr_{\mu}[X_{t+m}^{\mathrm{ave}} \in B ] \le  \frac{2}{s^{2}}  +\dl(2t,\mu) +
 \mathbb{E}[( \eta^{\pi}-  \eta)1_{G'}]   .
\end{equation}
\begin{equation}
\label{eq: lazytoave2}
|\mathbb{E}[( \eta-  \eta^{\pi})1_{G'}] |^{2}\le     s^{-1}
  \mathbb{E}_{\pi}\left[
\sup_{i \ge r \sqrt{s}   }i^{2}|\bigtriangleup A_{i}1_B |^2 \right] \le Cs^{-1} \Var_{\pi}1_B\le \frac{C}{s}.
\end{equation}
Consequently,
\begin{equation}
\label{eq: lazytoave4}
d_{\av}(t+m, \mu )\le \dl(2t, \mu )+\frac{2}{s^{2}}
+\sqrt{C/s} .
\end{equation}
\end{itemize}
\end{proposition}
\begin{proof}
We first note that (\ref{eq: cttolazy4}) follows from (\ref{eq: cttolazy1})-(\ref{eq: cttolazy2}) by maximizing over $B \subset \Omega $. We now prove (\ref{eq: cttolazy1}).
Let  $B \subset \Omega $. Let  $r,J$ and $m$ be as in Definition \ref{def: Dc}. By Fact \ref{fact: LD} and our assumption that $s \le e^{t}$ (which implies that $\epsilon:=r/t= 2\sqrt{2t^{-1}\log
s} \le3  $), $$\Pr[\xi_t \notin J ] \le \Pr[\xi_t<t- r]+  \Pr[\xi_t>t+ r]\le e^{-t \epsilon^2/2 }+e^{-\frac{t \epsilon^2/2}{(1+\epsilon/3)}}  =e^{-4 \log s }+e^{-\frac{4 \log s}{(1+\epsilon/3)}} \le 2s^{-2}.$$Hence $1-\Pr[G] \le \dc(t/2,\mu)+2s^{-2}$, which implies (\ref{eq: cttolazy1}), as  \[\pi(B)-\Pr_{\mu}[X_{t+m}^{\mathrm{L}} \in B ] \le 1-\Pr[G]+
 \Pr[G \cap \{ Z_{m}^{\mathrm{L,\pi}} \in B \} ]-\Pr[G \cap \{ Y_{t+m}^{\mathrm{L}} \in B \} ]
\] \[ =1-\Pr[G]+\mathbb{E}[(\eta^{\mathrm{L,\pi}}-\eta^{\mathrm{L}})1_{G}] . \]

  We now argue that for every $x\in \Omega$,
  \begin{equation}
\label{cttolazy3}
|\mathbb{E}[\eta-\eta^{\mathrm{L,\pi}}  \mid G, Y_{\xi_{t}}^{\mathrm{L}}=x=Z_0^{\mathrm{L,\pi}}] |\le  \sqrt{\frac{1}{s}}  \sup_{i \ge r\sqrt{s}  }i |\bigtriangleup P_{\laz}^{i}1_B(x) |.
\end{equation}
Indeed, for every $x\in \Omega$ and $j \in J $ 
$$\mathbb{E}[\eta^{\mathrm{L}}\mid\xi_t=j  ,Y_j^{\mathrm{L}}=x=Z_0^{\mathrm{L,\pi}}]=P_{\laz}^{t+m-j}1_B(x),$$  $$\mathbb{E}[\eta^{\mathrm{L,\pi}}\mid\xi_t=j  ,Y_j^{\mathrm{L}}=x=Z_0^{\mathrm{L,\pi}}]=P_{\laz}^{m}1_B(x) .$$
 Thus by the triangle inequality
\begin{equation}
\label{eq: triangleinequalitymainthm}
\begin{split}
& \mathbb{|E}[\eta^{\mathrm{L}}-\eta^{\mathrm{L,\pi}}\mid\xi_t=j  ,Y_j^{\mathrm{L}}=x=Z_0^{\mathrm{L,\pi}}]|=|P_{\laz}^{t+m-j}1_B(x)-P_{\laz}^{m}1_B(x)| \\ & \le 1_{j \neq t } \sum_{i=(t+m-j)
\wedge m}^{[(t+m-j) \vee m]-1}|\bigtriangleup P_{\laz}^{i}1_B(x)|.
\end{split}
\end{equation}
 Note that by the definition of $m=\lceil r( \sqrt{s}+1)\rceil$ and $J=[(t-r  )\vee 0,t+r ]$, for every $j \in J$ we have that $|j-t| \le r $ and $(t+m-j  )\wedge m \ge r\sqrt{s}$.   Whence,
\begin{equation*}
\begin{split}
&  1_{j \neq t } \sum_{i=(t+m-j) \wedge m}^{[(t+m-j) \vee m]-1}|\bigtriangleup P_{\laz}^{i}1_B(x)|  \le r \sup_{i \ge r \sqrt{s} }|\bigtriangleup P_{\laz}^{i}1_B(x)| \\ &  \le \frac{r}{ r \sqrt{s} }\sup_{i \ge r \sqrt{s} }i|\bigtriangleup P_{\laz}^{i}1_B(x)| =  \sqrt{s^{-1}}\sup_{i \ge r \sqrt{s} }i|\bigtriangleup P_{\laz}^{i}1_B(x)|.
\end{split}
\end{equation*}
Plugging this estimate in (\ref{eq: triangleinequalitymainthm}) and averaging over $j$ yields (\ref{cttolazy3}). 

Since $$ \mathbb{|E}[(\eta^{\mathrm{L}}-\eta^{\mathrm{L,\pi}})1_{G}]|
\le \mathbb{E}[ |\mathbb{E}[(\eta^{\mathrm{L}}-\eta^{\mathrm{L,\pi}})1_{G} \mid Z_0^{\mathrm{L,\pi}},\xi_t ] |], $$  averaging (\ref{cttolazy3}) over $Z_0^{\mathrm{L,\pi}}$, and using the fact that $\Pr[G \cap \{ Y_{\xi_{t}}^{\mathrm{L}}=x=Z_0^{\mathrm{L,\pi}} \} ] \le \pi(x)$, for all $x$, together with Jensen's inequality and (\ref{eq: maxtraingle1}), we get that
$$|\mathbb{E}[(\eta^{\mathrm{L}}-\eta^{\mathrm{L,\pi}})1_{G}] |^{2} \le  \frac{1}{s} \mathbb{(E}_{\pi}[ \sup_{i
\ge r \sqrt{s}  }i|\bigtriangleup P_{\laz}^{i}1_B | ])^{2}  \le  \frac{1}{s}\mathbb{E}_{\pi}[
\sup_{i \ge r \sqrt{s}  }i^{2}|\bigtriangleup P_{\laz}^{i}1_B |^2 ]  \le  Cs^{-1}\Var_{\pi}1_B \le C/s. $$
 \end{proof}
\section{Proof of Proposition \ref{prop: easydirectionintro}}
\label{sec: easy}
We start the section by stating a standard fact.\begin{claim}
\label{clm: independentTV}
Let $(\Omega,P,\pi)$ be a finite irreducible chain. Let $\mu \in \PP(\Omega)$. Let $(X_t)_{t \in \Z_{+}}$ be the discrete-time version of the chain. Let $T_1,T_2$ be independent $\Z_+$ valued random variables independent of $(X_{t})_{t \in \Z_+}$. Then $\| \Pr_{\mu}[X_{T_1+T_2} \in \cdot ]-\pi \|_{\mathrm{TV}} \le \| \Pr_{\mu}[X_{T_1 } \in \cdot ]-\pi
\|_{\mathrm{TV}}$, where $ \Pr_{\mu}[X_{T_1 } =y ]:=\sum_{t}\Pr[T_1=t]\Pr_{\mu}^{t}[X_t=y] $ and $ \Pr_{\mu}[X_{T_1+T_2 } =y ]:=\sum_{t}\Pr[T_1+T_{2}=t]\Pr_{\mu}^{t}[X_t=y]
$.
\end{claim}
\emph{Proof of Proposition \ref{prop: easydirectionintro}:}
Fix some $t >0$ and $0<s \le \sqrt{t}$. Denote $\tau:=t+s \sqrt{t} $.  We first prove \eqref{eq: ctfasterthanlazy1intro}. In the notation of the \emph{standard coupling}, $N_{\laz}(\tau) \sim \mathrm{Poisson}(2\tau) $ and $$\h_{\mu}^{\tau}-\pi= \sum_{k \ge 0}\Pr[N_{\laz}(\tau)=k](\Pr_{\laz,\mu} ^k-\pi).$$ By the triangle inequality, together with (\ref{eq: LD1}) and the fact that $\|\Pr_{\laz,\mu}^{k}-\pi
\|_{\TV} $ is non-decreasing in $k$ and bounded by 1,
\begin{equation*}
\begin{split}
& \| \h_{\mu}^{\tau}-\pi \|_{\TV} = \sum_{k \ge 0  }\Pr[N_{\laz}(\tau)=k]\|\Pr_{\laz,\mu}^{k}-\pi \|_{\TV} \le\Pr[N_{\laz}(\tau) < 2t ]+ \sum_{k \ge  2t }\Pr[N_{\laz}(\tau)=k]\|\Pr_{\laz,\mu}^{k}-\pi \|_{\TV}
\\ & \le \exp \left[- \frac{4s^{2}t}{2(2t+2s \sqrt{t})}\right] +\|\Pr_{\laz,\mu}^{2t}-\pi
\|_{\TV} \le d_{\laz}(2t,\mu)+e^{- \delta^{2}/2},
\end{split}
\end{equation*}
where in the last inequality we have used the assumption that $ s\le \sqrt{t}$. This concludes the proof of  \eqref{eq: ctfasterthanlazy1intro}. We now prove the first line in \eqref{eq: ctandlazyarefaster1intro}. We omit the second line in \eqref{eq: ctandlazyarefaster1intro} as its proof is analogous and as it essentially appears in \cite[Lemma 2.3]{peres2013mixing}.

 As above, denote $\tau:=t+s
\sqrt{t}$. Let $Y \sim \mathrm{Poisson}(2 \tau )$. Let $Z_1$ be a random variable whose conditional distribution,  given that $Y=n$, is  $ \mathrm{Bin}((n-1) \vee 0 ,1/2)$. Let $\eta$ be a Bernoulli random variable with mean $1/2$, independent of $Z_1$ and $Y$. Set $Z:=Z_1+\eta 1_{Y>0} $. Let   $(X_t)_{t \in \Z_{+}}$ be the  discrete-time version of the chain with $X_0 \sim \mu$.   Pick $Y$, $Z_{1}$,  $\eta$ and $(X_t)_{t \in \Z_{+}}$ to be jointly independent. Note that the conditional distribution of $Z$,   given that $Y=n$, is   $ \mathrm{Bin}(n,1/2)$. Hence by Poisson thinning $Z \sim \mathrm{Poisson}(\tau) $ and so $X_{\tau}^{\ct}\sim X_{Z}$.
\medskip

  Let $T:=t+\eta$. Then $Z = (T+Z_1-t)1_{Y>0}$.  Thus $Z 1_{Z_{1} \ge t} =(T+(Z_1-t)_{+})1_{Z_1 \ge t}$, where $a_+:=a \vee 0$ (since $Z_1 \ge t$ implies that $Y>0$ and $Z_1-t=(Z_1-t)_+$). Consequently,
\begin{equation}
\label{eq: X_z1}
\|\Pr_{\mu}(X_{Z} \in \cdot )-\Pr_{\mu}(X_{T+(Z_1-t)_{+}} \in \cdot
) \|_{\TV} \le \|Z-[T+(Z_1-t)_{+}] \|_{\TV} \le \mathbb{P} [Z_1 < t  ].
\end{equation}
By (\ref{eq: LD1}) and the assumption $s \le \sqrt{t}$,
\begin{equation}
\label{eq: X_z2}
  \mathbb{P}[Z_1 < t ] \le  \mathbb{P}[Z \le t  ]        \le\exp \left[- \frac{s^{2}t}{2(t+s \sqrt{t})}\right] \le  e^{- s^{2}/4}.
\end{equation}
Finally, by Claim \ref{clm: independentTV}, in conjunction with (\ref{eq: X_z1})-(\ref{eq: X_z2}), we get that
\begin{equation*}
\label{eq: X_z3}
\begin{split}
& \dc(t+s
\sqrt{t},\mu)=\| \Pr_{\mu}[X_{Z } \in \cdot ]-\pi
\|_{\mathrm{TV}} \le \\ & \|\Pr_{\mu}(X_{Z} \in \cdot )-\Pr_{\mu}(X_{T+(Z_1-t)_{+}} \in \cdot
) \|_{\TV}+ \|\Pr_{\mu}(X_{T+(Z_1-t)_{+}}
\in \cdot
) - \pi  \|_{\TV} \\ & \le e^{- s^{2}/4}+ \|\Pr_{\mu}(X_{T}
\in \cdot
) - \pi  \|_{\TV}=d_{\av}(t,\mu)+e^{- s^{2}/4}.   \quad \qed
\end{split}
\end{equation*}

\section{Proof of Theorem \ref{thm: window}}
\label{s: 5}
Assume that there is a continuous-time cutoff with a window $w_n$. Fix some $0<\epsilon < 1/4$. By Propositions \ref{prop: AFintro} (first inequality) and \ref{prop: sqrt} (second inequality)
$$\ave^{(n)}(\epsilon) \le \mix^{(n)}(\epsilon/2)+C_1(\epsilon) \sqrt{\mix^{(n)}(\epsilon/2)} \le  \mix^{(n)}(\epsilon/2)+C_2(\epsilon)w_n. $$
By Propositions \ref{prop: easydirectionintro} (first inequality) and \ref{prop: sqrt}
(second inequality) we have that
$$-\ave^{(n)}(1-\epsilon) \le -\mix^{(n)}(1-\epsilon/2)+C_{3}(\epsilon) \sqrt{\mix^{(n)}} \le -\mix^{(n)}(1-\epsilon/2) +C_4(\epsilon) w_n.   $$
Hence
$$\ave^{(n)}(\epsilon)-\ave^{(n)}(1-\epsilon) \le\mix^{(n)}(\epsilon/2)-\mix^{(n)}(1-\epsilon/2)+C_5(\epsilon)w_n \le C_6(\epsilon)w_n,$$
as desired. Now assume that the sequence of averaged chains exhibits a cutoff with a window $\tilde w_n$. By Proposition \ref{prop: easydirectionintro} 
$$ \mix^{(n)}(\epsilon) \le \ave^{(n)}(\epsilon/2)  +C_7(\epsilon) \sqrt{\mix^{(n)}} . $$
By Propositions \ref{prop: AFintro} we have that
$$ -\mix^{(n)}(1-\epsilon) \le -\ave^{(n)}(1-\epsilon/2)  +C_{8}(\epsilon) \sqrt{\mix^{(n)}} .   $$
Hence
$$ \mix^{(n)}(\epsilon)-\mix^{(n)}(1-\epsilon) \le \ave^{(n)}(\epsilon/2)-\ave^{(n)}(1-\epsilon/2) +C_9(\epsilon) \sqrt{\mix^{(n)}}\le C_{10}(\epsilon)(\tilde w_n \vee  \sqrt{\mix^{(n)}}) ,$$
as desired. \qed
\section{Example}
\label{sec: examples}
In this section we consider an example which demonstrates that the assertions of Theorems \ref{thm: AF} and \ref{thm: AF2} and of Proposition \ref{prop: AFintro} are in some sense nearly sharp. For notational convenience we suppress the dependence on $n$ in some of the notation below. Throughout this section we write $c_0,c_1,c_2,\ldots $ for positive absolute constants, which are sufficiently small to guarantee that a certain inequality holds.

 Equation (\ref{eq: mainconverseexample}) below  resembles our main results apart from the fact that below the direction of the inequality is reversed, and the exponent of $s$ in the error term of the middle term in  (\ref{eq: mainconverseexample})  (which decays like an inverse polynomial in $s$) is larger (compared to the corresponding exponent in Theorem \ref{thm: AF}; similarly, the error term on the RHS of  (\ref{eq: mainconverseexample}) is similar to the one appearing in Theorem \ref{thm: AF2}, that is to $\psi_{\alpha,C_{2}}(\dc(t))-\dc(t)$). 

\begin{example}
Fix some $0 < \alpha \le 1/2 $. Let $n \in \N $ be such that $s=s_{n,\alpha}  :=\lceil n^{0.5+\alpha} \rceil \ge 2$.  Consider a nearest-neighbor random walk on the interval $\{0,1,2,\ldots,2n+1\}$, with a  bias towards state $2n+1$, whose transition matrix is
given by $P(0,1)=1$, $P(2n+1,2n)=1-\frac{1}{3s}$, $$P(i,i)=\begin{cases} \frac{1}{3s} &  i \ge 2 n-2s, \\
0 & \text{otherwise}. \\
\end{cases}$$
Finally, $P(i,i+1)=3P(i,i-1)$ for all $1 \le i \le 2n$ and is given by
 $$P(i,i+1)=\begin{cases}\frac{3}{4}-\frac{1}{4s} & i \ge  2 n-2s, \\
3/4 & \text{otherwise}. \\
\end{cases}$$ 
By Kolmogorov's cycle condition, this chain is reversible. 
Both the sequence of the associated
continuous-time chains and the sequence of the associated averaged chains  exhibit cutoff around time $4n$ with a cutoff window of size $\sqrt{n}$. In particular, prior to time $4n-s $ the worst-case total variation distance from stationarity of both chains tends to 1 as $n$ tends to infinity.
Moreover, it is not hard to show that$$\dc(4n+s)=(1 \pm o(1))  \h_0[T_{2n+1} >4n+s] \le e^{-c_3 s^2/n} \le e^{-c_3n^{2\alpha}}.  $$
Conversely, we now show that for $t=4n+s$, we have that
\begin{equation}
\label{eq: mainconverseexample}
d_{\av}(t+s) \ge\dc(t)+\frac{c_1}{s} \ge \dc(t)+ \frac{c_2}{\left[ \log (1/\dc(t))\right]^{\frac{1+2\alpha}{4\alpha}}}.
\end{equation} 
The second inequality in (\ref{eq: mainconverseexample}) follows from the choice $s =\lceil n^{\frac{1+2\alpha}{2}} \rceil $ together with $\dc(t)=\dc(4n+s) \le  e^{-c_3n^{2\alpha}} $. We now prove the first inequality in (\ref{eq: mainconverseexample}).  

Consider the sets $\mathrm{Even}:=\{2i:0 \le i \le n \}$, $\mathrm{Odd}:=\{2i+1:0 \le i \le n \} $ and $B:=\{i: i \ge 2n-2s \} $. It is easy to see that $\pi(B) \ge 1-2^{-(2s+1)} $ and that
\begin{equation}
\label{eq: piofeven}
0 \le \pi(\mathrm{Even}) -1/2 \le \frac{\pi(2n-2s)}{3s} \le 2^{-2s}.
\end{equation}
In order to prove (\ref{eq: mainconverseexample}), we shall show that
\begin{equation}
\label{eq: even2}
A_{t+s}(0,\mathrm{Even}) \ge \frac{1}{2} + \frac{c_1}{s}.
\end{equation}
Let  $(X_k)_{k=0}^{\infty}$ be the discrete-time chain with  $X_0=0$.  Note that $T_{2n-2s}$ is even, deterministically. 
If both $X_{4n+2s}$ and $X_{4n+2s+1}$ lie in $B$, we define $$T:=\min \{k:T_{2n-2s} \le k \le 4n+2s \text{ and } X_{\ell} \in B \text{ for all }k \le \ell \le 4n+2s+1  \}. $$   Otherwise, set $T=0$. It is easy to see that $\Pr[T=0] \le C e^{-c_4s^2/n} $ and that
\begin{equation}
\label{eq: T=0}
\frac{1}{2} \Pr_{0}[X_{4n+2s} \in \mathrm{Even} \mid T=0 ]+\frac{1}{2}\Pr_{0}[X_{4n+2s+1} \in \mathrm{Even} \mid T=0]= 1/2.
\end{equation}
Moreover, conditioned on $T>0$, the number of returns to state $2n-2s$ by time $4n+2s$ has an exponential tail. Using this fact, it is not hard to verify that
\begin{equation}
\label{eq: Tnot0}
\begin{split}
& \min_{0 \le r \le 4s } \Pr[T \text{ is even} \mid T \neq 0,4n+2s-T_{2n-2s} = 2r  ] \ge 1-\frac{c_5}{s}.
\\ & \Pr[4n+2s-T_{2n-2s}>8s \mid T \neq 0 ] \le e^{-c_6s^2/n}.
\end{split}
\end{equation}
Consider the projected chain $(Y_{k})_{k=0}^{4n+2s+1-T}$ (conditioned on $T \neq 0$) on $\Omega:=\{ \pm 1\}$ defined via  $Y_{k}:=1_{T+k \in \mathrm{Even} }-1_{T+k \in \mathrm{Odd} }$.  This two state chain whose transition matrix is given by $P=\begin{pmatrix}\frac{\lambda}{2} & 1-\frac{\lambda}{2} \\
1-\frac{\lambda}{2} & \frac{\lambda}{2} \\
\end{pmatrix}$, where  $\lambda:=\frac{2}{3s} $, satisfies $P\begin{pmatrix}1 \\
-1 \\
\end{pmatrix}=(\lambda -1) \begin{pmatrix}1 \\
-1 \\
\end{pmatrix}$. Using the spectral decomposition it is easy to verify that $A_{k}(1, 1)=\frac{1}{2}+\frac{(\lambda-1)^k
\lambda}{4}$. Note that if $k \le 8s$ then for even $k$'s we have that $ 0 \le A_{k}(1, 1)-\frac{1}{2}= \Theta(s^{-1})$ and for odd $k$'s $ 0 \le \frac{1}{2}- A_{k}(1, 1)= \Theta(s^{-1})$. 

Applying this for $k=r$ when $T=4n+2s-r>0$, in conjunction with (\ref{eq: T=0})-(\ref{eq: Tnot0}) yields (\ref{eq: piofeven}) by averaging over $4n+2s-T$ and bounding separately the contribution of all even times (i.e.~$4n+2s-T=2k$, $k \le 4s$) and of all odd times, which are bounded from above by $8s$ . We leave the details as an exercise.
  \end{example}
\section*{Acknowledgements}
We are grateful to David Aldous, Riddhipratim Basu and Allan Sly for many useful discussions. In addition we  want to thank Riddhipratim Basu and Emma Cohen for reading previous drafts of this work and suggesting many improvements to the presentation.

...

\nocite{*}
\bibliographystyle{plain}
\bibliography{cutoff}
\end{document}